\documentclass[12pt,twoside]{article}
\usepackage{amssymb,amsmath,amsthm,latexsym}
\usepackage{amsfonts}
\usepackage{graphicx}
\newtheorem{thm}{Theorem}[section]

\newtheorem{cor}[thm] {Corollary}
\newtheorem{lem} [thm]{Lemma}

\newtheorem{rmk}[thm] {Remark}
\newtheorem{defn}[thm]{Definition}
\numberwithin{equation}{section}

\def\ni{\noindent}
\begin{document}
\label{'ubf'}
\setcounter{page}{1} 

\markboth {\hspace*{-9mm} \centerline{\footnotesize \sc
    On the Characteristic Polynomial of Skew Gain Graphs}
                 }
                { \centerline {\footnotesize \sc
             Shahul Hameed K Roshni T Roy Soorya P Germina K A  } \hspace*{-9mm}
               }
\begin{center}
{
       {\huge \textbf{On the Characteristic Polynomial of Skew Gain Graphs 
                               }
       }
\\

\medskip

Shahul Hameed K \footnote{\small  Department of
Mathematics, K M M Government\ Women's\ College, Kannur - 670004,\ Kerala,  \ India.  E-mail: shabrennen@gmail.com} 
Roshni T Roy \footnote{\small Department of Mathematics, Central University of Kerala, Kasaragod - 671316,\ Kerala,\ India.\ Email:roshnitroy@gmail.com}
Soorya P \footnote{\small  Department of Mathematics, Central University of Kerala, Kasaragod - 671316,\ Kerala,\ India.\ Email: sooryap2017@gmail.com}
Germina K A \footnote{\small  Department of Mathematics, Central University of Kerala, Kasaragod - 671316,\ Kerala,\ India.\ Email: srgerminaka@gmail.com}
}
\end{center}
\newcommand\NEPS{\operatorname{NEPS}}
\thispagestyle{empty}
\begin{abstract}
\ni Gain graphs are graphs where the edges are given some orientation and labeled with the elements (called gains) from a group so that gains are inverted when we reverse the direction of the edges. Generalizing the notion of gain graphs, skew gain graphs have the property that the gain of a reversed edge is the image of edge gain under an anti-involution. In this paper, we deal with the adjacency matrix of skew gain graphs with involutive automorphism on a field of characteristic zero and their charactersitic polynomials. Spectra of some particular skew gain graphs are also discussed. Meanwhile it is interesting to note that weighted graphs are particular cases of skew gain graphs.
\\
---------------------------------------------------------------------------------------------\\
\end{abstract}
\section{Introduction and Basic Results}
In this paper, we provide a general expression for computing the coefficients of the characterstic polynomials of skew gain graphs, with involutive automorphism on $\mathbb{F}^\times$ where $F$ is a field of characteristic zero, which are the generalization of the same in the case of gain and signed graphs and discuss the spectra of some skew gain graphs. Before, we delve into the details of skew gain graphs, we require some defintions mainly that of an anti-involution. We denote a group by $\Gamma$ and when we use matrices the elements are taken from the multiplicative group $F^\times$ where $F$ is a field of characteristic zero. For details regarding graphs, signed graphs, gain graphs and skew gain graphs, the reader may refer to \cite{fh,fh1,j1,shkg,tz3}. All the underlying graphs in this article are simple. We call a function $f:\Gamma\rightarrow \Gamma$ to be an \emph{involution} if $f(f(x))=x$ for all $x\in \Gamma$.  A function $f:\Gamma\rightarrow \Gamma$ is called an \emph{anti-homomorphism} if $f(xy)=f(y)f(x)$ for all $x,y\in\Gamma$. Note that for an abelian group an anti-homomorphism is always a homomorphism. An involution $f:\Gamma\rightarrow \Gamma$ which is an anti-homomorphism is called an \emph{anti-involution}. We use $\mathrm{Inv}(\Gamma)$ to denote the set of all anti-involutions on $\Gamma$. To make the discussion self contained, we provide the proofs of the results relating to involutions and anti-involutions. 
\begin{lem}\label{l1} Every involution is bijective.
\end{lem}
\begin{proof} Let $f:\Gamma\rightarrow \Gamma$ be an involution. $f$ is injective, since $f(x)=f(y)\Longrightarrow f(f(x))=f(f(y))\Longrightarrow x=y$.
$f$ is surjective, since given $y\in\Gamma$ taking $x=f(y)\in\Gamma$, $f(x)=f(f(y))=y$. Hence $f$ is bijective.
\end{proof}
\begin{lem}\label{l2} Every anti-homomorphism $f:\Gamma\rightarrow \Gamma$ satisfies the following:\\
(i) $f(1)=1$ \\ (ii) $f(x^{-1})=(f(x))^{-1}$.
\end{lem}
\begin{proof} (i) Since $1=1.1$, $f(1)=f(1).f(1)\Longrightarrow f(1)=1.$ \\
(ii) $x.\ x^{-1}=1=x^{-1}.\ x\Longrightarrow f(x^{-1}).\ f(x)=f(1)=1=f(x).\ f(x^{-1})$ which completes the proof.
\end{proof}
\begin{lem}\label{l3} $f:\Gamma\rightarrow \Gamma$ is an anti-involution if and only if there exists an involution $g:\Gamma\rightarrow \Gamma$ which is an automorphism such that $f(x)=g(x^{-1})$ for all $x\in \Gamma$.
\end{lem}
\begin{proof} Let $f$ be an anti-involution. Define $g:\Gamma\rightarrow \Gamma$ such that $g(x)=f(x^{-1})$ for all $x\in \Gamma$. Then,
\begin{align*} g(xy)&=f((xy)^{-1})=f(y^{-1} x^{-1})\\
                    &=f(x^{-1})f(y^{-1})\\
                    &=g(x)g(y)
\end{align*}  which shows that $g$ is a homomorphism. $g$ is injective, since $f$ is and inverse of an element in $\Gamma$ is unique. To show that $g$ is surjective, take $y\in\Gamma$. Then $f$ being surjective, there exists $x\in\Gamma$ such that $f(x)=y$. Then $g(x^{-1})=f(x)=y$. Also $g$ is an involution since
\begin{align*} g(g(x))&=g(f(x^{-1}))\\
                      &=g(f(x)^{-1})\\
                      &=f(f(x))\\
                      &=x.
\end{align*}
Converse follows easily from the definition of $g$.
\end{proof}
\begin{lem}Let $\Gamma$ be an abelian group. If $f\in \mathrm{Inv}(\Gamma)$, then $g:\Gamma \rightarrow \Gamma$ defined by $g(x)=xf(x)$ is a homomorphism. 
\end{lem}
\begin{proof}
\begin{align*} g(xy)&=xyf(xy)\\
                      &=xf(x)yf(y)\\
                      &=g(x)g(y)
\end{align*} 
\end{proof}
\ni Now it is time to define what are skew gain graphs. Though plenty of literature can be cited dealing with the structures like graphs, signed graphs and gain graphs the detials of which are beautifully collected by Zaslavsky in ~\cite{tz3}, we could trace out only the works of J. Hage and T. Harju ~\cite{j1,j2} who defined the skew gain graphs. Our attempt to analyse the structure using matrices is the first of its kind in that direction. From now onwards, the notation $\overrightarrow{E}$ stands for the collection of oriented edges such that for an edge $uv\in E$ of a graph, we have the oriented edges $\overrightarrow{uv}$ and $\overrightarrow{vu}$ in $\overrightarrow{E}$.  
\begin{defn}[\cite{j1}]\rm{ Let $G=(V,\overrightarrow{E})$ be a graph with some prescribed orientation for the edges and $\Gamma$ be an arbitrary group. If $f\in \mathrm{Inv}(\Gamma)$ then the \emph{skew-gain graph} $\Phi_f=(G,\Gamma,\varphi,f)$ is such that the \emph{skew gain function} $\varphi:\overrightarrow{E}\rightarrow \Gamma$ satisfies $\varphi(\overrightarrow{vu})=f(\varphi(\overrightarrow{uv}))$.}
\end{defn}
\ni To quote some examples of skew gain graphs, note that every graph is a skew gain graph where the group $\Gamma$ is chosen as the multiplicative group $\{1\}$ and the function $f$ as the identity funtion. Signed graphs and gain graphs are particular cases of skew gain graphs by suitable choices of the groups and involution.
Another exciting idea is that weighted graphs are skew gain graphs with weights chosen from a group and the function $f$ is the identity function.
\ni The \emph{skew gain}, $\varphi(C)$, of a cycle $C:v_0v_1 \dots v_nv_0,$ is the product $\varphi(v_0v_1)\varphi(v_1v_2)\dots\varphi(v_nv_0)$ of the skew gains of its edges. Also, when the underlying graph is a path $P_n$ or a cycle $C_n,$ we call the corresponding structures to be \emph{skew gain path} or  \emph{skew gain cycle}, respectively.\\

\section{Adjacency matrix and Characteristic polynomial of skew gain graphs}
Let $F$ be a field of characteristic zero. We define the function $g:F^\times\rightarrow F^\times$ by $g(x)=xf(x)$ where $f\in \mathrm{Inv}(F^\times).$ In the case of gain graphs this $g$ ceases to be the trivial homomorphism.
Given a skew gain graph $\Phi_f=(G,F^\times,\varphi,f)$ its adjacency matrix $A(\Phi_f)=(a_{ij})_n$ is defined as the square matrix of order $n=|V(G)|$ where \\
$a_{ij} =
\left\{
\begin{array}{ll}
\varphi(v_iv_j)  & \mbox{if } v_i\sim v_j \\
0 & \mbox{otherwise }
\end{array}
\right.$ \\ such that whenever $a_{ij}\neq 0$,  
$a_{ji}=f(a_{ij})$. We denote the charactersitic polynomial of the skew gain graph $\Phi_f$ by $\Psi(\Phi_f,x)= \det(xI-A(\Phi_f))$.  We define, as usual, a subgraph of a graph as an \emph{elementary subgraph}\  \cite{fh1}, if its components consist only $K_2$ or cycles. In the following formulae, we take sum over all elementary subgraphs $L\in\mathfrak{L}_{i}$ where $\mathfrak{L}_{i}$ denotes the collection of all elementary subgraphs $L$ of order $i$. For $i=0,1$, we take $a_i(\Phi_f)=1, 0$ respectively in order to avoid confusion. Also the notation $K(L)$ is used to denote the number of components in $L$.
\begin{thm}\label{gen} If $\Phi_f =(G,F^\times,\varphi,f)$ is a skew gain graph where $G=(V,E)$ is a graph of order $n$, and if $\Psi(\Phi_f,x)=\displaystyle\sum_{i=0}^{n}a_i(\Phi_f)x^{n-i}$ then
\begin{equation}\label{eq1} a_i(\Phi_f)=\displaystyle \sum_{L\in\mathfrak{L}_{i}}(-1)^{K(L)}\Big(\displaystyle\prod_{K_2\in L}\displaystyle\prod_{\overrightarrow{e}\in K_2}g(\varphi(\overrightarrow{e}))\Big)\displaystyle\prod_{C\in L}(\varphi(C)+f(\varphi(C)))
\end{equation}
\end{thm}
\begin{proof}
As we deal with only simple graphs $G,$ using the usual Laplacian expansion of the determinant, the coefficients $ a_i(\Phi_f)$ are given by $(-1)^i$ times the principal minors of order $i$. Hence $ a_i(\Phi_f)=\displaystyle \sum_{L\in\mathfrak{L}_{i}}(-1)^{K(L)}\displaystyle\prod_{\overrightarrow{e}\in E(L)}\varphi(\overrightarrow{e})$. But by the definition of elementary subgraph $L$, it has only two types of components namely $K_2$ or cycles. Thus, $\displaystyle\prod_{\overrightarrow{e}\in E(L)}\varphi(\overrightarrow{e})$ simplifises to $\Big(\displaystyle\prod_{K_2\in L}\displaystyle\prod_{\overrightarrow{e}\in K_2}g(\varphi(\overrightarrow{e}))\Big)\displaystyle\prod_{C\in L}(\varphi(C)+f(\varphi(C)))$
\end{proof}
\ni It is of particular interest to note that the charactersitic polynomials of skew gain paths and skew gain cycles can be expressed in terms of the matching sets as in the following theorems. In what follows, the notation $\mathcal{M}_k(G)$ denotes the set of all matchings $M$ of $G$ having exacly $k$ independent edges. Proofs of both the following results are omitted because they follow easily from Equation~\eqref{eq1} and the fact that the elementary subgraphs here  will form the concerned matching sets in the case of paths and for the cycle, apart from that, the largest ordered elementary subgraphs will include the underlying cycle also.
\begin{cor}\label{path} If $\Phi_f(P_n)=(P_n,F^\times,\varphi,f)$ is a  skew gain path, then the characteristic polynomial  $\Psi(\Phi_f(P_n),x)=x^n+\displaystyle\sum_{k=1}^{\lfloor\frac{n}{2}\rfloor}(-1)^{k}a_{2k } x^{n-2k}$ has the coefficients given by 
	\begin{equation*}
	a_{2k }= \displaystyle\sum_{M\in \mathcal{M}_k(G)}\prod_{e\in M}g(\varphi(e)) 
	\end{equation*}
\end{cor}
\begin{cor}\label{cycle} If $\Phi_f(C_n)=(C_n,F^\times,\varphi,f)$ is a  skew gain cycle, then the characteristic polynomial $\Psi(\Phi_f(C_n),x)=x^n+\displaystyle\sum_{k=1}^{\lfloor\frac{n}{2}\rfloor}(-1)^{k}a_{2k }x^{n-2k}-(\varphi(C)+f(\varphi(C)))$ with the coefficients $a_{2k }$ given by
	\begin{equation*}
	a_{2k }= \displaystyle\sum_{M\in \mathcal{M}_k(G)}\prod_{e\in M}g(\varphi(e)) 
	\end{equation*}
\end{cor}
\begin{cor}\label{detpath}
	
	$\det(A(\Phi_f(P_n)))$\begin{equation*}
= \left\{
\begin{array}{rl} 0 & \text{, if } n\equiv 1 \pmod {2},\\
(-1)^{\frac{n}{2}}\displaystyle\sum_{M\in \mathcal{M}_{\frac{n}{2}}(G)}\prod_{e\in M}g(\varphi(e)),  & \text{ if } n\equiv 0 \pmod {2}
\end{array} \right.
	\end{equation*}
\end{cor}
\begin{proof}
	Note that the determinant of a matrix is $(-1)^n$ times the constant term in the characteristic polynomial. Then the result follows easily from Corollary~\ref{path}.
\end{proof}
\begin{cor}\label{detcycle}
	
	$\det(A(\Phi_f(C_n)))$\begin{equation*}
	= \left\{
	\begin{array}{rl} \varphi(C)+f(\varphi(C)) & \text{, if } n\equiv 1 \pmod {2},\\
	(-1)^{\frac{n}{2}}\displaystyle\sum_{M\in \mathcal{M}_{\frac{n}{2}}(G)}\prod_{e\in M}g(\varphi(e))-(\varphi(C)+f(\varphi(C))),  & \text{ if } n\equiv 0\pmod{2}	\end{array} \right.
	\end{equation*}
\end{cor}
\begin{proof}
The result follows from Corollary~\ref{cycle} with the facts mentioned above.
\end{proof}
\begin{rmk}\rm{
An important observation from Corollary~\ref{detcycle} is that 	$\det(A(\Phi_f(C_n)))$ is never zero for an odd cycle as no multiplicative anti-involution satisfies $f(x)=-x$.}
\end{rmk}
\ni Indeed Corollary~\ref{path} is applicable to any skew gain graphs with underlying graphs as trees. In the case of a skew gain graph with a bipartite graph as the underlying graph, we note generally that some coefficients will be zero as shown below. i.e., for an odd ordered underlying graph, there will be only terms having odd powers of $x$ and for even ordered case only even powers of $x$. 
\begin{thm}\label{tree}
	If $\Phi_f(G)=(G,F^\times,\varphi,f)$ is skew gain graph, where $G$ is a bipartite graph,  then the characteristic polynomial is   $\Psi(\Phi_f(G),x)=x^n+\displaystyle\sum_{k=1}^{\lfloor\frac{n}{2}\rfloor}a_{2k} x^{n-2k}$ where the coefficients are given by \\
	\begin{equation*}
	a_{2k}=\displaystyle \sum_{L\in\mathfrak{L}_{2k}}(-1)^{K(L)}\Big(\displaystyle\prod_{K_2\in L}\displaystyle\prod_{\overrightarrow{e}\in K_2}g(\varphi(\overrightarrow{e}))\Big)\displaystyle\prod_{C\in L}(\varphi(C)+f(\varphi(C)))
	\end{equation*}
\end{thm}
\begin{proof}
The coefficient of $x^{n-i}$ depends on the collection $\mathfrak{L}_{i}$ of elementary subgraphs of order $i$. When $i$ is odd, components of $L\in\mathfrak{L}_{i}$ must contain a cycle of odd order. Otherwise, if for some $L\in\mathfrak{L}_{i}$, all its components are $K_2$ or cycles of even order implies $L$ has even order, a contradiction. Now since $G$ is bipartite, it has no odd cycles. Hence the coefficients corresponding to $x^{n-i}$ become $0$ when $i$ is odd.	
\end{proof}
\ni Also Corollary~\ref{cycle} can be extended to the case of skew gain graphs with unicyclic graphs as the underlying graphs as follows.
\begin{thm}
	If $\Phi_f(U_n)=(U_n,F^\times,\varphi,f)$ is a skew gain unicyclic graph with unique cycle $C_p$ then the characteristic polynomial  $\Psi(\Phi_f(U_n),x)= x^n+\displaystyle\sum_{i=1}^{\lfloor\frac{n}{2}\rfloor}a_{2i} x^{n-2i}+\displaystyle\sum_{i=0}^{\lfloor\frac{n-p}{2}\rfloor}b_{(p+2i)} x^{n-(p+2i)}$ with the coefficients $a_{2i}$ are given by \\
	
	$ a_{2i}=  \displaystyle\sum_{M\in \mathcal{M}_{i}(U_n)}(-1)^{i}\prod_{e\in M}g(\varphi(e))$ and $b_{p+2i}$ are given by\\
	
	$b_{p+2i}=  [\varphi(C_p)+f(\varphi(C_p))]\displaystyle\sum_{M\in \mathcal{M}_{i}(U_n-C_p)}(-1)^{i+1}\prod_{e\in M}g(\varphi(e))$
	
\end{thm}

\begin{proof}
	
	As the coefficient of $x^{n-i}$ depends on the collection $\mathfrak{L}_{i}$ of elementary subgraphs of order $i,$ letting $C_p$ to be the unique cycle in $U_n,$ all the elementary subgraphs of order less than $p$ contain only $K_2$ as their components. Also, the coefficient of $x^{n-i}$, for $i \geq p$ depend on the presence or absence of $C_p$. \\
	Let $\mathfrak{L}_{i}(C_p)$ denote the collection of elementary subgraphs of order $i$ containing $C_p$ as a component. Then by Theorem ~\ref{gen}, the coefficient of $x^{n-i}$ that does not depend on $C_p$ will reduces to,		
	\begin{equation*}a_i
	= \left\{
	\begin{array}{rl} 	
	\displaystyle\sum_{M\in \mathcal{M}_{i}(G)} (-1)^{i}\prod_{e\in M}g(\varphi(e)),  & \text{ if $i$ is even } \\ 
	0,  & \text{ if $i$ is odd} 
	\end{array} \right.
	\end{equation*}
	The coefficients that depend on $\mathfrak{L}_{k}(C_p)$ are those of  $x^{n-k},$ where $k\geq p+2i, i = 0,1,\dots ,{\lfloor\frac{n-p}{2}\rfloor}.$ Thus, the coefficient of $x^{n-(p+2i)}$  will become, 
	\begin{equation*}	 
	b_{p+2i}=  [\varphi(C_p)+f(\varphi(C_p))]\displaystyle\sum_{M\in \mathcal{M}_{i}(U_n-C_p)}(-1)^{i+1}\prod_{e\in M}g(\varphi(e))
	\end{equation*}
\end{proof}
\section{ Spectra of some skew gain graphs}
The eigenvalues of the adjacency matrices, counting the multiplicities, of a skew gain graph are called the eigenvalues or spectra of that skew gain graph. If $\lambda_1,\lambda_2,\cdots, \lambda_{k}$ are the eigenvalues with algebraic multiplicities $\alpha_1,\alpha_2,\cdots, \alpha_k$, we express the same with the notation $\begin{pmatrix} \lambda_1 &\lambda_2 &\cdots & \lambda_k\\
\alpha_1&\alpha_2&\cdots& \alpha_k\end{pmatrix}$. Basic details regarding spectra of graphs can be had from~\cite{spec1}. As we are dealing with matrices over arbitary fields, the characteristic roots are taken from its algebraic closure.  	
Let us begin the discussion by giving spectra of skew gain graphs with a star $K_{1,n}$ as the underlying graph. The following Corollary follows from Theorem~\ref{tree} which is used to give the spectra of the aforesaid skew gain graph in Theorem~\ref{star}. 
\begin{cor}
	If $\Phi_f(G)=(G,F^\times,\varphi,f)$ is a skew gain graph where $G= K_{1,n}$ is a star graph then the characteristic polynomial is 
	\begin{equation*} 
	\Psi(\Phi_f(G),x)=x^{n+1} -\Big( \displaystyle\sum_{e \in E(G)} g(\varphi(e))\Big )x^{n-1}.
	\end{equation*}
\end{cor}
\begin{thm}\label{star}
	If $\Phi_f(G)=(G,F^\times,\varphi,f)$ is a skew gain graph where $G= K_{1,n}$ is a star of order $n+1$, then the spectrum of $\Phi_f(G)$ is $\begin{pmatrix}  -\sqrt{{\displaystyle\sum_{e \in E(G)}g(\varphi(e))}} & \sqrt{{\displaystyle\sum_{e \in E(G)}g(\varphi(e))}}  & 0\\
	1 & 1 & n-1
	\end{pmatrix}$
\end{thm}

\ni Next theorem gives the spectra of skew gain graphs with a double star graph as the underlying graph.
\begin{thm}
	If $\Phi_f(G)=(G,F^\times,\varphi,f)$ is a skew gain graph where $G$ is a double  star graph of order $n$ then the characteristic polynomial  $\Psi(\Phi_f(G),x)=x^n-  a_2x^{n-2}+ a_4x^{n-4}$ where $a_2=  \displaystyle\sum_{e \in E(G)} g(\varphi(e))$ and $a_4 = \displaystyle\sum_{M\in \mathcal{M}_2(G)}\prod_{e\in M}g(\varphi(e)).$ Hence the spectrum of $\Phi_f(G)$ is\\ $\begin{pmatrix}  -\sqrt{\frac{a_2 - \sqrt{{a_2}^2 - 4a_4}}{2}} & -\sqrt{\frac{a_2 + \sqrt{{a_2}^2 - 4a_4}}{2}} & \sqrt{\frac{a_2 - \sqrt{{a_2}^2 - 4a_4}}{2}} & \sqrt{\frac{a_2 + \sqrt{{a_2}^2 - 4a_4}}{2}} & 0\\
	1 & 1 & 1 & 1 & n-4
	\end{pmatrix}$. 
\end{thm}
\begin{proof}
	The expression for the characteristic polynomial directly follows from Theorem \ref{tree}. Thus the characteristic polynomial is 
	\begin{eqnarray*}
	\Psi(\Phi_f(G),x)=x^n-  a_2x^{n-2}+ a_4x^{n-4}\\
	= x^{n-4}\big(x^4-  a_2x^2+ a_4\big)
	\end{eqnarray*} 
	where $a_2=  \displaystyle\sum_{e \in E(G)} g(\varphi(e))$ and $a_4 = \displaystyle\sum_{M\in \mathcal{M}_2(G)}\prod_{e\in M}g(\varphi(e)).$ 
	Hence the eigenvalues are $\pm\sqrt{\frac{a_2 \pm \sqrt{{a_2}^2 - 4a_4}}{2}}$ with multipilicity one each and $0$ with multiplicity $n-4.$ 
\end{proof}
	
\ni To deal with the case of skew gain graphs with underlying graphs as a complete bipartite graph $K{m,n} \ ,$ first we give certain well known results from matrix theory.
\begin{lem}[\cite{fz}]\label{schur}\ If $M=\begin{pmatrix} A & B\\
	                                             C & D
	                             \end{pmatrix}$ and if $ C$  and $D$ commute, then $\det(M)=\det(AD-BC).$                

\end{lem}
\begin{lem}[\cite{fz}] \label{schur2}(Schur's lemma)\ If $M=\begin{pmatrix} A & B\\
	C & D
	\end{pmatrix}$ and if $ A $ is invertible then $\det(M)=\det(A)\det(D-CA^{-1}B)$. Also if $D$ is invertible $\det(M)=\det(D)\det(A-BD^{-1}C).$                 
	
\end{lem}
\begin{thm}\label{bpt}
	If  $\Phi_f(G)=(G,F^\times,\varphi,f)$ is a skew gain graph where $G$ is a complete bipartite graph $K_{m,n}$ with $m<n,$ then 0 is an eigenvalue of $\Phi_f(G)$  with multiplicity atleast $n-m.$
\end{thm}
\begin{proof}
	Let $G$ be complete bipartite $K_{m,n}$  with bipartitions $V_1$ and $V_2.$ The coefficient of $x^{(n+m)-i}$ depends on the collection $\mathfrak{L}_{i}$ of elementary subgraphs of order $i$. For any $L\in\mathfrak{L}_{i},$ $\frac{i}{2}$ vertices in $L$ are from $V_1$ and $\frac{i}{2}$ are from $V_2$ since the components of $L$ are $K_2$ or even cycles in $G$. Thus the maximum value that $i$ can assume is $2\min\{m,n\} = 2m$. Then the least power of $x$ is $(m+n) - 2m = n-m  \neq 0$ and hence $0$ is an eigenvalue of $\Phi_f(G)$  with multiplicity atleast $n-m$.
\end{proof}

\ni Recall that matching number is the cardinality of maximum matching in $G$. Using this for a skew gain graph with tree as its underlying graph, we have the following.

\begin{thm}
	If  $\Phi_f(G)=(G,F^\times,\varphi,f)$ is a skew gain graph where $G$ is a tree of order $n$ having matching number $t < \dfrac{n}{2}$, then $0$ is an eigenvalue $\Phi_f(G)$ with multiplicity atleast $n-2t$.
\end{thm}
\begin{proof}
	The coefficient $x^{n-i}$ in the characteristic polynomial $G$ depends on matchings having exactly $\frac{i}{2}$ edges. Since $G$ have matching number $t$, the maximum value for $\frac{i}{2}$ is $t$ and hence maximum value of $i$ is $2t$. Then the least power of $x$ is $n-2t \neq 0$ and hence $0$ is an eigenvalue $\Phi_f(G)$ with multiplicity atleast $n-2t$. 
\end{proof}
\ni To continue the discussion, we define for a matrix $B=(a_{ij})\in M_{m\times n}(F), B^{f}=(b_{ij})\in M_{m\times n}(F)$ where $f\in \mathrm{Inv}(F^{\times})$ as follows.
$b_{ij} =
\left\{
\begin{array}{ll}
f(a_{ij})  & \mbox{if } a_{ij}\neq 0 \\
0 & \mbox{otherwise }
\end{array}
\right.$\\
Also for this matrix $B\in M_{m\times n},$ $B^{\#}\in M_{n\times m}$ is defined as $ B^{\#}=(B^{f})^T.$ Note that for the anti-involution $f_1\in \mathrm{Inv}({F}^\times)$ defined by $f_1(x)=x,$ the matrix $B^{\#}=B^T$ and in the case of the anti-involution $f_2\in \mathrm{Inv}(\mathbb{C}^\times)$ defined by $f_2(z)=\overline{z},$ $B^{\#}=B^*,$ the conjugate transpose of $B.$ 
\begin{thm}
	If the adjacency matrix of  $\Phi_f(K_{m,n})=(K_{m,n} \ ,\mathbb{F}^\times,\varphi,f)$, where $m\leq n$, is expressed as $A(\Phi_f(K_{m,n}))=\begin{pmatrix}
	O  &  B\\
	B^{\#}  & O 
	\end{pmatrix} $,
	then the non zero eigenvalues of $\Phi_f(K_{m,n})$ are $\lambda$ such that $\lambda^2$ is an eigenvalue of $BB^{\#}$.
	
\end{thm}
\begin{proof}
	The characteristic equation for the matrix 	
	$\begin{pmatrix}
	O  &  B\\
	B^{\#}  & O 
	\end{pmatrix}$ is \\
	\begin{center}
		$\det \begin{pmatrix}
		x I_{m \times m}  &  -B_{m \times n}\\
		-B^{\#}_{n \times m}  &  xI_{n \times n} 
		\end{pmatrix} = 0$
	\end{center}
	For $m=n$, $x I$ and $B^{\#}$ commutes and hence using Lemma~\ref{schur}, characteristic equation becomes	$\det (x^2I- BB^{\#})=0$. This implies the eigenvalues of $\Phi_f(K_{m,m})$ are $\lambda$ such that $\lambda^2$ is an eigenvalue of $BB^{\#}$.\\
	For $m<n$, by Theorem ~\ref{bpt}, $0$ is an eigenvalue with multiplicity atleast $n-m$.
	Now for $x \neq 0$, $xI_{n \times n}$ is invertible and hence using the Schur's Lemma~\ref{schur2},\\
	\begin{align*}
	\det \begin{pmatrix}
	x I_{m \times m}  &  -B_{m \times n}\\
	-B^{\#}_{n \times m}  &  xI_{n \times n} 
	\end{pmatrix}
	& = \det ( x I_{n \times n}) \det(x I_{m \times m}- B. \dfrac{1}{x}I_{n \times n}.B^{\#}) \\
	& = x^{n-m}.\det (x^2I_{m \times m}- BB^{\#}).
	\end{align*}
	This implies the non zero eigenvalues are $\lambda$ such that $\lambda^2$ is an eigenvalue of $BB^{\#} $.
	
\end{proof}

\end{document}